\documentclass[preprint]{imsart}

\RequirePackage[OT1]{fontenc}
\RequirePackage{amsthm,amsmath}
\RequirePackage{natbib}
\RequirePackage[colorlinks,citecolor=blue,urlcolor=blue]{hyperref}
\usepackage{amsthm,amsmath,natbib, mathtools, bbm, amsfonts}
\startlocaldefs
\numberwithin{equation}{section}
\theoremstyle{plain}
\newtheorem{thm}{Theorem}[section]
\newtheorem{lemma}{Lemma}[section]
\newtheorem{proposition}{Proposition}[section]

\newtheorem{definition}{Definition}[section]
\newtheorem{remark}{Remark}[section]
\endlocaldefs

\begin{document}

\include{Safikhani_Reference.bib}

\begin{frontmatter}

\title{Spectral Conditions for Equivalence of Gaussian Random Fields with Stationary Increments  
\thanks{The research is supported in part by NSF grants DMS-1612885 and DMS-1607089.}
}
\runtitle{Equivalence of Gaussian Random Fields   }

\begin{aug}
\author{\fnms{Abolfazl} \snm{Safikhani}\ead[label=e1]{as5012@columbia.edu}},
\author{\fnms{Yimin} \snm{Xiao}\ead[label=e2]{xiao@stt.msu.edu}}


\affiliation{Department of Statistics\\Columbia University}

\address{Department of Statistics,\\
Columbia University\\
1255 Amsterdam Ave, 928\\
New York, NY 10027\\
\printead{e1}\\
\phantom{E-mail:\ }}


\address{Department of Statistics and Probability,\\
Michigan State University\\
619 Red Cedar Road, C413 Wells Hall\\
East Lansing, MI 48824-1027\\
\printead{e2}\\ }
\end{aug}



\begin{abstract}

This paper studies the problem of equivalence of Gaussian measures induced by Gaussian random fields (GRFs) with stationary increments and proves a 
sufficient condition for the equivalence in terms of the behavior of the spectral measures at infinity. The main results extend those of Stein (2004), Van Zanten (2007, 2008)
and are applicable to a rich family of nonstationary space-time models with possible anisotropy behavior. 

\end{abstract}

\end{frontmatter}

\section{Introduction}

Space-time models have become increasingly popular in scientific studies such as geology, climatology, geophysics, environmental 
and atmospheric sciences, etc (\cite{Chiles_2012}, \cite{Cressie_1993} and \cite{Stein_1999}). Gaussian random fields (GRFs) are 
ubiquitous in space-time modeling due to the prevalence of central limit theorems and the mathematical/computational amenability 
of the multivarite normal distributions. Most of the parametric models proposed in the literature are GRFs with specific parametric 
covariance structure (see  \cite{Cressie_Huang_1999},  \cite{Gneiting_2002},  \cite{Stein_2005}, and \cite{Gneiting_Genton_Guttorp_2007} 
for rich families of space-time covariance functions). One of the main objectives in statistics then is to find consistent estimates for 
the parameters,  and finally use them for prediction of the underlying random field at unobserved locations. Given a parametric family 
of Gaussian random fields, an important question is to determine whether all the parameters are consistently estimable. First step to 
answer this question demands an investigation on the equivalence or singularity of the corresponding Gaussian measures  induced 
by this family of GRFs on their space of sample functions, since if two sets of parameters in the Gaussian models give equivalent 
Gaussian distributions, then it is impossible to find consistent estimators for these parameters involved regardless of the method 
chosen for estimation (see for example \cite{Zhang_2004}  for a discussion on inconsistent estimation in Mat\'ern covariance 
functions under the framework of fixed domain asymptotics). Another  application of equivalence of Gaussian measures comes 
from covariance structure misspecification, and its effect on spatial interpolation (see \cite{Stein_1999}). 
Therefore, finding explicit conditions for deciding whether two Gaussian random fields induce
equivalent Gaussian distributions on their spaces of sample functions has direct impact 
in evaluating the prediction error in interpolation of spatial data, and thus proving asymptotically optimal prediction under misspecified 
covariance structure. There are other applications of equivalence and perpendicularity of GRFs in spatial modeling. For example, we 
refer 
to \cite{Furrer_Genton_Nychka_2006}, \cite{Kaufman_Schervish_Nychka_2008}, 
\cite{safikhanicovariance} and \cite{safikhani2015nonstationary} for the application in covariance tapering.

Equivalence of Gaussian measures is a classical problem in probability theory that has been studied since the 1950's. 
We refer to the books \cite{Gikhman_Skorokhod_2004}, \cite{Ibragimov_Rozanov_1978}, \cite{Yadrenko_1983}, 
\cite{mandrekar2015stochastic} and references therein for systematic accounts. Necessary and sufficient conditions 
for the equivalence of Gaussian measures induced by stationary Gaussian processes in terms of their mean and 
covariance functions are given in \cite{Ibragimov_Rozanov_1978}. Their extensions to stationary isotropic Gaussian 
random fields are proved in \cite{Skorokhod_Yadrenko_1973} and \cite{Yadrenko_1983}. Among these results, the 
explicit criterion for equivalence of stationary Gaussian processes in terms of their spectral densities  (cf. Theorem 17 
on p. 104 of \cite{Ibragimov_Rozanov_1978}) is particularly convenient to apply.  This criterion has been extended 
by \cite{Skorokhod_Yadrenko_1973} and \cite{Yadrenko_1983} to stationary isotropic Gaussian random fields.

However, investigation on the equivalence of nonstationary GRFs has been limited to some special cases. For instance, 
we refer to \cite{Cheridito_2001} on mixed fractional Brownian motion, \cite{Baudoin_Nualart_2003} on Volterra processes,
\cite{Sottinen_Tudor_2006} on Gaussian random fields that are equivalent to fractional Brownian sheets, \cite{Stein_2004} 
on a family of intrinsic random functions with power law generalized covariance functions (including fractional Brownian fields),
\cite{VanZanten_2007, VanZanten_2008} on Gaussian processes with stationary increments, and \cite{Xue_2011} on certain 
Gaussian random fields with stationary increments.  

Our work is mainly motivated by \cite{Stein_2004} and by \cite{VanZanten_2007, VanZanten_2008} 
where explicit sufficient conditions for the equivalence of Gaussian processes with stationary increments in terms 
of their spectral densities similar to the criterion in \cite{Ibragimov_Rozanov_1978} for the stationary case have
been established. The main purpose of this paper is to extend their results to the setting of Gaussian random fields 
with stationary increments which may have different regularities in each direction. Besides of theoretical interest, 
our results are applicable to anisotropic nonstationary space-time Gaussian models.

The rest of the paper is organized as follows. We start Section \ref{pr} by introducing some useful Hilbert spaces connected to 
the frequency domain, and study their structure. In Section \ref{main results}, we state the main result of the paper, which is 
sufficient conditions for equivalence of GRFs with stationary increments using the tail behavior of their spectral densities. In the 
last section, we apply the main results to a rich family of anisotropic nonstationary spatio-temporal Gaussian models.

\section{Preliminary}\label{pr}

Let $ X = \lbrace X_t : t \in \mathbb{R}^{d}\rbrace $ be a centered GRF with stationary increments. The covariance structure of $X$ 
is fully described by \cite{Yaglom_1957}. For simplicity, we assume that $ X(0)=0$
and the covariance function of $X$ can be written as 
\begin{equation}\label{spectral representation}
C(t,s)= \mathbf{E} (X_t X_s) = \int_{\mathbb{R}^{d}}\big( e^{i\langle t,\lambda \rangle} - 1\big)
\big(e^{-i\langle s, \lambda \rangle} - 1\big)\,F(d\lambda), 
\end{equation}
where $ F $ is a non-negative symmetric measure on $ \mathbb{R}^{d}\setminus\lbrace 0 \rbrace $, called the 
spectral measure of $ X $, that satisfies
\begin{equation}\label{Con}
	\int_{\mathbb{R}^{d}} {\frac{{\vert \lambda \vert}^{2}}{1+{\vert \lambda \vert}^{2}}} F(d\lambda)<\infty.
\end{equation}
It follows from (\ref{spectral representation}) that $X$ has the following spectral representation:
\begin{equation}
X(t) \,{\buildrel d \over =}\, \int_{\mathbb{R}^{d}}\!\big(e^{i\langle t,\lambda \rangle} - 1\big)\,W(d\lambda), 
\end{equation}
where $ W $ is a complex-valued Gaussian random measure with mean 0 and control measure $ F $.

If the spectral measure $F$ is absolutely continuous with respect to the Lebesgue measure on $ \mathbb{R}^d $, 
we will call its Radon-Nikodym derivative, denoted by $ f(\lambda) $, the spectral density of $X$. We will give conditions 
for the equivalence of GRFs with stationary increments in terms of their spectral densities, but first, we recall the 
definition of equivalence of GRFs.  

\begin{definition}
For a fixed set $ D \subseteq \mathbb{R}^d $, we call two GRFs $ X = \lbrace X_t : t \in \mathbb{R}^{d}\rbrace $, 
$ Y = \lbrace Y_t : t \in \mathbb{R}^{d}\rbrace $ \textit{equivalent on $D$} if they induce equivalent measures\footnote{Two 
measures defined on the same measurable space are called equivalent if they are mutually absolutely continuous with 
respect to each other.} on 
the measurable space $ \left( \mathbb{R}^D, \mathcal{B}\left( \mathbb{R}^D \right) \right) $, in which 
$ \mathcal{B}\left( \mathbb{R}^D \right) $ is the $ \sigma $-field generated by the cylinder subsets\footnote{A 
cylinder subset of $\mathbb{R}^D$ is a set of the form $\{f \in \mathbb{R}^D:\, f(t^1) \in B_1, \ldots,  f(t^n) \in B_n\}$, where 
$t^1, \ldots, t^n \in D$ and $B_1, \ldots, B_n$ are Borel sets in $\mathbb R$.}
of $\mathbb{R}^D$. 
Moreover, we call $X$ and $Y$ \textit{locally equivalent} if they are equivalent on all bounded subsets of $\mathbb{R}^d $. 
\end{definition}

The spectral representation \eqref{spectral representation} makes an important bridge between the problem of 
equivalence of GRFs and the description of the space generated by the linear combinations of the kernel functions.  
For that purpose, in this section we 
define for a fixed and bounded set $ D \subseteq \mathbb{R}^d $, an incomplete Hilbert space $ \mathcal{L}^e_D 
= \mbox{span} \left\lbrace e_t(\lambda):= e^{i\langle t, \lambda \rangle} - 1: t \in D \right\rbrace $ with the inner product 
$$ {\langle e_t, e_s \rangle}_F = \int_{\mathbb{R}^d}\!e_t(\lambda)\overline{e_s(\lambda)}\,F(d\lambda). $$ 
We denote the closure  of $ \mathcal{L}^e_D $ in $ L^2(F) $ by $ \mathcal{L}_D(F) $. Also, for $ T > 0 $, denote by
 $\Pi_T  = [-T,T]^d$  the cube with side $2T$. 

Observe that the functions in $ \mathcal{L}^e_D$ are entire functions defined on $ \mathbb{C}^d$ (see \cite{Ronkin_1974} 
for definition and more properties), and they are of finite exponential type.  Recall that an entire function $\varphi$ on 
$\mathbb{C}^d $ is called of finite exponential type if 
\[
\limsup_{r \to \infty}  \frac 1 r \max_{\|z\| = r} \log \big|\varphi(z)\big| < \infty.
\]
However, in general, the elements in the completed Hilbert space, $ \mathcal{L}_D(F) $, may not have the same properties 
as the functions in $ \mathcal{L}^e_D$.  
This problem is discussed in details in \cite{Pitt_1973, Pitt_1975}. 
In this paper, we assume that the spectral measure $F $ has a  density function $f(\lambda)$ that satisfies the following 
condition: 
\begin{verse}
$ {\rm (C1)}$ \, There exist constants $ c, k, \eta > 0 $ such that $ f(\lambda) \geq \frac{c}{{\vert \lambda \vert}^{\eta}}$
for all $\lambda \in \mathbb R^d$ with $ \vert \lambda \vert > k$.  
\end{verse}

This assumption on the spectral density will imply the elements in $ \mathcal{L}_D(F) $ to be entire functions of finite 
exponential type. These properties enable us to apply the Paley-Wiener type theorems to get nice description of the 
elements in the Hilbert space $ \mathcal{L}_D(F) $ for $D = \Pi_T$. 

The next two lemmas will prove these statements. The following lemma is taken 
from \cite{Xiao_2007} and we state it here again for completeness.

\begin{lemma}\label{xiao}
Suppose that the spectral density $ f $ satisfies (C1). Then for fixed $T>0 $, there exists positive constants $ C $ 
and $ M $ such that for all functions $ \phi $ of the form 
\begin{equation}\label{elem}
\phi (\lambda) = \sum_{k=1}^{n} {a_k \big( e^{i\langle t^k, \lambda \rangle} - 1 \big)},
\end{equation}
where $ a_k \in \mathbb{R} $ and $ t^k \in \Pi_T $, we have for all $ z \in \mathbb{C}^d $ 
\begin{equation}\label{entire}
\vert \phi(z) \vert \leq C\, {\Vert \phi \Vert}_F \, \exp\lbrace M \vert z \vert \rbrace. 
\end{equation}
Moreover, for fixed $ C_1>0 $, there exists a positive constant $ C_2 $ such that for all functions of the form 
\eqref{elem}, we have 
\begin{equation}\label{Shwartz}
\vert \phi(z) \vert \leq C_2 \, \vert z \vert \, {\Vert \phi \Vert}_F
\end{equation}
for all $ \vert z \vert \leq C_1 $.  
\end{lemma}

One can use \eqref{entire} to define the limiting functions in $ \mathcal{L}_{\Pi_T}(F) $ in such a way that they also 
satisfy both \eqref{entire} and  \eqref{Shwartz}. We will prove this in the next Lemma. 

\begin{lemma}\label{Pitt}
Suppose that the spectral density $ f $ satisfies (C1).  Then, for each $ T>0 $, the space $\mathcal{L}_{\Pi_T}(F)$ 
consists of the restriction to $ \mathbb{R}^d $ of entire functions on $ \mathbb{C}^d $ of finite exponential type. 
Moreover, \eqref{Shwartz} holds for all functions $\phi \in  \mathcal{L}_{\Pi_T}(F)$.
\end{lemma}

\begin{proof}
The idea of the proof is similar to \cite{Pitt_1975}, p. 304. Take a sequence $\phi_n \in \mathcal{L}^e_{\Pi_T} $, 
such that $ {\Vert \phi_n - \phi \Vert}_F \rightarrow 0 $ for some $\phi \in \mathcal{L}_{\Pi_T}(F) $. Then, it is a 
Cauchy sequence in $ \L^2(F) $, and using \eqref{entire}, we get
\begin{equation}
\vert \phi_n(z) - \phi_m(z) \vert \leq C\, {\Vert \phi_n - \phi_m \Vert}_F \, \exp\lbrace M \vert z \vert \rbrace
\end{equation}
This means for each fixed $ z \in \mathbb{C}^d $, the sequence $ \{\phi_n(z), n \ge 1\} $ is a Cauchy sequence 
in $ \mathbb{C} $. So, it is convergent and, moreover, the convergence is locally uniform. Denote the limit by 
$\widetilde{\phi}(z)$. Now, since limit in $ L^2(F) $ sense implies the almost everywhere convergence for a 
subsequence, $\phi = \widetilde{\phi} $ a.e. with respect to $F$. From now on, we will take $ \widetilde{\phi} $ 
as our favorite version of the limits of functions in $ \mathcal{L}^e_{\Pi_T} $. Therefore, the elements in the 
space $ \mathcal{L}_{\Pi_T}(F) $, are not only the $ L^2(F) $ limits of functions in $ \mathcal{L}^e_{\Pi_T}$, 
but also the pointwise limits as well. Thus, both \eqref{entire} and \eqref{Shwartz} are true for all the elements 
in $ \mathcal{L}_{\Pi_T}(F) $. The only thing left to prove is that these functions $ \widetilde{\phi}$ are entire 
functions on $\mathbb{C}^d $. But this is true since any element of the space $ \mathcal{L}_{\Pi_T}(F) $ is 
the locally uniform limit of functions of the form \eqref{elem} which are obviously entire functions, and thus, 
they are entire functions as well (This is called the Weirerstrass Theorem; see, for example, Proposition 
2.8 on p. 52 of \cite{Ebeling_2007}).      
\end{proof}

This lemma shows that if the spectral density satisfies the assumption (C1), we can complete the space 
$\mathcal{L}^e_{\Pi_T} $ in such a way that the resulting functions are locally uniform limits of entire 
functions, and hence, they are entire functions  of finite exponential type. Furthermore, since 
\eqref{entire} is true for all the elements in the Hilbert space $ \mathcal{L}_{\Pi_T}(F) $, we can see that the 
point evaluators, i.e. the functionals on $ \mathcal{L}_{\Pi_T}(F) $ of the form $ \phi \mapsto \phi(z) $ for 
each fixed $ z \in \mathbb{C}^d $ are bounded operators. Now, we can apply the Riesz Representation 
Theorem (See \cite{Halmos_1957}, Theorem 3. p. 31) to prove that the space $ \mathcal{L}_{\Pi_T}(F) $ 
is a Reproducing Kernel Hilbert Space (RKHS) in the sense of \cite{Aronszajn_1950}. This means that 
there exists a function $ K_T(\cdot , \cdot): \, \mathbb{R}^d \times \mathbb{R}^d \to  \mathbb{C} $ such that (i) 
$ K_T(\omega, \cdot) \in \mathcal{L}_{\Pi_T}(F) $ for all $ \omega \in \mathbb{R}^d $, and (ii)
for every $ \phi \in \mathcal{L}_{\Pi_T}(F) $ and $ \omega \in \mathbb{R}^d $, we have the following 
kernel property
\begin{equation}\label{RKHS}
{\left\langle \phi, K_T(\omega, \cdot)\right\rangle}_F = \int_{\mathbb{R}^d}\!\phi(\lambda)
\overline{K_T(\omega, \lambda)}\,F(d\lambda) = \phi(\omega).
\end{equation}

Also, it is worthwhile to mention that the set of all functions $ \lbrace K_T(\omega, \cdot ): \omega \in 
\mathbb{R}^d \rbrace $ is dense in $ \mathcal{L}_{\Pi_T}(F) $ (To see this, note that if $ \phi \in 
\mathcal{L}_{\Pi_T}(F) $ is orthogonal to $ K_T(\omega, \cdot ) $ for all $ \omega \in \mathbb{R}^d$, 
then $ \phi(\omega) = {\langle \phi,K_T(\omega, \cdot ) \rangle}_F = 0 $, which implies $ \phi = 0 $). 
Futhermore, for all $ \psi \in L^2(F) $, the function
\begin{equation}
\omega \mapsto {\langle \psi, K_T(\omega, .)\rangle}_F
\end{equation}  
is the orthogonal projection of $ \psi $ on $ \mathcal{L}_{\Pi_T}(F) $ (See the proof in \cite{Aronszajn_1950}, 
p. 345). We denote this projection by $ \pi_{\mathcal{L}_{\Pi_T}(F)} \psi $.

Finding explicit forms of the reproducing kernels is not an easy job. However, in order to prove the results in 
Section \ref{main results}, we only need to establish upper bounds for the growth rate of the diagonal elements of the 
reproducing kernels at origin and also at infinity. The following proposition proves an important growth rate
 for the diagonal elements in the reproducing kernels.
 
\begin{proposition}\label{zero bhv}
Suppose that the spectral density $ f(\lambda) $ of $ F$ satisfies (C1). Then, for fixed $ T>0 $ and
 $ C_1>0 $, there exists a positive constant $ C_2 $ such that 
\begin{equation}\label{zero behavior}
K_T(\omega, \omega) \leq C_2 \, {\vert \omega \vert}^2
\end{equation}
for all $\omega \in \mathbb R^d$ with $\vert \omega \vert < C_1$.
\end{proposition}
\begin{proof}
Since for any fixed $ \omega \in \mathbb{R}^d $, $ K_T(\omega, .) \in \mathcal{L}_{\Pi_T}(F) $, we can apply 
Lemma \ref{Pitt} to these functions. It follows from \eqref{Shwartz} that
\[
\vert K_T(\omega, \lambda)\vert \leq C_2 \, \vert \lambda \vert \, {\Vert K_T(\omega, .)\Vert}_F
=  C_2 \, \vert \lambda \vert {\left( K_T(\omega, \omega)\right)}^{1/2}
\]
for all $ \omega \in \mathbb{R}^d $ and $ \lambda \in \mathbb{C}^d $ with $ \vert \lambda \vert < C_1 $. 
By taking $ \lambda = \omega $,  we obtain the desired result.
\end{proof}

We also need to define another Hilbert space based on the tensor product of the elements in 
$\mathcal{L}_{\Pi_T}(F) $. For this purpose, first we define $ \mathcal{L}^e_{\Pi_T} \otimes \mathcal{L}^e_{\Pi_T} $ 
to be the span of functions $ (e_t \otimes e_s)(\omega, \lambda) := e_t(\omega) \overline{e_s(\lambda)} $ 
with $ t, s \in \Pi_T $. Now, denote by $ \mathcal{L}_{\Pi_T}(F) \otimes \mathcal{L}_{\Pi_T}(F) $ the closure in 
$ L^2(F \otimes F) $ of the space $ \mathcal{L}^e_{\Pi_T} \otimes \mathcal{L}^e_{\Pi_T} $. According to Theorem 
1 on p. 361 of \cite{Aronszajn_1950},  the new Hilbert space $ \mathcal{L}_{\Pi_T}(F) \otimes \mathcal{L}_{\Pi_T}(F)$ 
is also a RKHS with reproducing kernel
\begin{equation}
\left( (\omega_1, \lambda_1), (\omega_2, \lambda_2) \right) \mapsto K_T(\omega_1, \omega_2) 
\overline{K_T(\lambda_1, \lambda_2)}.
\end{equation}  
This implies that for $ \psi \in \mathcal{L}_{\Pi_T}(F) \otimes \mathcal{L}_{\Pi_T}(F) $,
\begin{equation}
{\langle \psi, K_T(\omega, .)\otimes K_T(\lambda, .) \rangle}_{F \otimes F} = \psi(\omega, \lambda).
\end{equation}

We finish this section by a lemma stating that the norm of the elements in spaces $ \mathcal{L}_{\Pi_T}(F) $ 
depends essentially on the tail behavior of the spectral measure $ F $. 

\begin{lemma}\label{equivalence}
Suppose $ f_0 $ and $ f_1 $ are two spectral densities satisfying the condition $ (C1) $, and further 
$ f_0(\lambda) \asymp f_1(\lambda) $ as $ |\lambda| \rightarrow \infty $. Then, $ \mathcal{L}_{\Pi_T}(f_0) 
= \mathcal{L}_{\Pi_T}(f_1) $, and further there exist positive constants $ C_3 $ and $ C_4 $ such that
\begin{equation*}
C_3 \, {\Vert\phi\Vert}_{f_1} \leq {\Vert\phi\Vert}_{f_0} \leq C_4 \, {\Vert\phi\Vert}_{f_1}
\end{equation*} 
for all $ \phi \in \mathcal{L}_{\Pi_T}(f_0)$.
\end{lemma}  
\begin{proof}
Suppose $ g \in  \mathcal{L}_{\Pi_T}(f_0) $. There exists functions $ g_n $ of the form \eqref{elem} such that 
$ { \Vert g - g_n \Vert }_{f_0} \rightarrow 0 $ as $ n\rightarrow \infty $. Now, using Lemma \ref{Pitt}, we get
\begin{equation}\label{Eq:g}
\begin{split}
{ \Vert g - g_n \Vert }^2_{f_1}
&= \int_{\mathbb{R}^d}\!{|g(\lambda) - g_n(\lambda)|}^2f_1(\lambda)\,d\lambda\\
& \leq  C_2 \, { \Vert g - g_n \Vert }^2_{f_0} \int_{|\lambda|<C_1}\!|\lambda|^2 f_1(\lambda)\,d\lambda\\
& + \, C \, { \Vert g - g_n \Vert }^2_{f_0}
\end{split}
\end{equation}
which implies $ g \in \mathcal{L}_{\Pi_T}(f_1) $. Thus $ \mathcal{L}_{\Pi_T}(f_0) \subseteq \mathcal{L}_{\Pi_T}(f_1) $. 
Similarly we have $ \mathcal{L}_{\Pi_T}(f_1) \subseteq \mathcal{L}_{\Pi_T}(f_0)$. Finally one can see that \eqref{Eq:g}
holds if we replace  $ g - g_n$ by $  \phi \in \mathcal{L}_{\Pi_T}(f_0)$ This leads to the second part of the lemma.
\end{proof}

\section{Main Results}\label{main results}

In this section, we study the equivalence of GRFs with stationary increments, and clarify its connection to the 
Hilbert spaces constructed in Section \ref{pr}. In particular, the role of the reproducing kernels of the RKHS 
$ \mathcal{L}_{\Pi_T}(F) $ will be emphasized. We start this section by an extension of Theorem 5 on p. 84 in 
\cite{Ibragimov_Rozanov_1978} (Theorem 1 on p. 149 in  \cite{Yadrenko_1983}) for stationary Gaussian 
processes (fields) to Gaussian random fields with stationary increments.  Some extensions of the criteria 
for equivalence of stationary Gaussian processes (fields) have also been obtained by \cite{VanZanten_2007,
VanZanten_2008} and \cite{Xue_2011}. 
The following theorem is an extension of Theorems 3.3.9 and 3.3.10 in \cite{Xue_2011} and also Theorem 4.3 
in \cite{VanZanten_2007}.

\begin{thm}\label{one} 
Two centered GRFs with stationary increments and spectral measures $ F_0 $ and $ F_1 $ are equivalent on 
$D$ if and only if:
\begin{itemize}
\item[(i)]\, $ {\Vert \phi \Vert}_{F_0} \asymp {\Vert \phi \Vert}_{F_1}, \forall \phi \in \mathcal{L}^e_D $, and
\item[(ii)]\, There exists a function $ \psi \in \mathcal{L}_D (F_0) \otimes \mathcal{L}_D (F_0) $ such that for all $ t, s \in D $
\begin{equation}\label{ess}
	{\langle e_t, e_s \rangle}_{F_0} - {\langle e_t, e_s \rangle}_{F_1} = {\langle \psi, e_t \otimes e_s \rangle}_{F_0 \otimes F_0}.		
\end{equation}  
\end{itemize}
\end{thm}

\begin{proof}
The proof is essentially a reconstruction of the proof of the Theorem 5, p. 84 of \cite{Ibragimov_Rozanov_1978}, and is 
given here for the sake of completeness. The 
starting point is however, Theorem 4 on p. 80 of the same reference. The proof there can be adapted to our context 
with little change since it only involves the ``entropy distance" between the Gaussian measures, and thus is true for 
general GRFs (See also \cite{Chatterji_1978}, Theorem 4.1, 4.4  pp. 180-185 ). After doing so, we get that two GRFs with stationary 
increments and spectral measures $ F_0 $ and $ F_1 $ are equivalent on $ D \subseteq \mathbb{R}^d $, if and only if
$$ {\Vert \phi \Vert}_{F_0} \asymp {\Vert \phi \Vert}_{F_1}, \ \ \ \forall \, \phi \in \mathcal{L}^e_D, $$
and $ \Delta = I - A^*A $ is a Hilbert-Schmidt operator in $ \mathcal{L}_D(F_0) $, in which $ I $ is the identity operator 
on $ \mathcal{L}_D(F_0) $, and $ A: \mathcal{L}_D(F_0) \mapsto \mathcal{L}_D(F_1) $ with $ A\phi = \phi $ for all 
$ \phi \in \mathcal{L}_D(F_0) $. Now, since $ \Delta $ is a self-adjoint operator, if it is also a Hilbert-Schmidt operator, 
by the Spectral Theorem (See \cite{Dunford_1963}, Corollary 5 p. 905 ), we can conclude that there exists an orthonormal 
basis for $ \mathcal{L}_D(F_0) $ consisting of the eigenvectors of $ \Delta $, denoting them by $ \phi_j , j = 1, 2, . . . $, 
with corresponding eigenvalues $ \lambda_j, j = 1, 2, . . . $ with $ \sum_{j}\lambda^2_j < \infty  $. Note that we can write 
$ \sum_{j}\lambda^2_j = \sum_{j,k} { \langle \Delta\phi_j, \phi_k \rangle}^2_{F_0} $. The square root of this quantity is 
called the Hilbert-Schmidt norm. This norm doesn't depend on the choice of the orthonormal basis (See \cite{Dunford_1963}, 
Lemma 2, p. 1010). Therefore, we can rephrase Theorem 4 in the following form: two GRFs with stationary increments 
and spectral measures $ F_0 $ and $ F_1 $ are equivalent on $ D $, if and only if, $ {\Vert \phi \Vert}_{F_0} 
\asymp {\Vert \phi \Vert}_{F_1}, \forall \phi \in \mathcal{L}^e_D $, and $  \sum_{j,k} { \langle \Delta\phi_j, 
\phi_k \rangle}^2_{F_0} < \infty $ for any orthonormal basis for $ \mathcal{L}_D(F_0) $.

Now, take an arbitrary orthonormal basis for $ \mathcal{L}_D(F_0) $, $ \phi_1, \phi_2, . . . $, and suppose 
$ \sum_{j,k} { \langle \Delta\phi_j, \phi_k \rangle}^2_{F_0} < \infty $. Define $ \psi_0(\omega, \lambda) =   
\sum_{j,k} { \langle \Delta\phi_j, \phi_k \rangle}_{F_0}\phi_j(\omega)\overline{\phi_k(\lambda)} $. We can see that 
$ {\Vert \psi_0 \Vert}^2_{F_0 \otimes F_0} = \sum_{j,k} { \langle \Delta\phi_j, \phi_k \rangle}^2_{F_0} < \infty $, and 
thus by the form of $ \psi_0 $, it's clear that it belongs to $ \mathcal{L}_D(F_0) \otimes \mathcal{L}_D(F_0) $. Also, 
observe that
\begin{equation*}
\begin{split}
{\langle \psi_0, \phi_j \otimes \phi_k  \rangle}_{F_0 \otimes F_0}
&= {\langle \Delta\phi_j, \phi_k \rangle}_{F_0}\\
&= {\langle (I - A^*A)\phi_j, \phi_k \rangle}_{F_0}\\
&= {\langle \phi_j, \phi_k \rangle}_{F_0} - {\langle \phi_j, \phi_k \rangle}_{F_1}.
\end{split}
\end{equation*}
This shows that \eqref{ess} holds for orthonormal basis of the space $ \mathcal{L}_D(F_0) $. Therefore, by continuity of inner 
product \eqref{ess} will be true for all the elements in $ \mathcal{L}_D(F_0) $, especially for $ e_t $ and $ e_s $ when $ t, s \in D $.  

Conversely, suppose there exists a function $ \psi_0 \in \mathcal{L}_D(F_0) \otimes \mathcal{L}_D(F_0) $, such that 
$ {\langle \phi_j, \phi_k \rangle}_{F_0} - {\langle \phi_j, \phi_k \rangle}_{F_1} 
= {\langle \psi_0, \phi_j \otimes \phi_k  \rangle}_{F_0 \otimes F_0} $ for an orthonormal basis $ \phi_j $'s for 
$ \mathcal{L}_D(F_0) $. Then, we have 
\begin{equation*}
\begin{split}
\sum_{j,k} {\langle \Delta \phi_j, \phi_k \rangle}^2_{F_0}
&= \sum_{j, k}{\left( {\langle \phi_j, \phi_k \rangle}_{F_0} - {\langle \phi_j, \phi_k \rangle}_{F_1}\right)}^2\\ 
&= \sum_{j, k} {\langle \psi_0, \phi_j \otimes \phi_k  \rangle}^2_{F_0 \otimes F_0}\\
& \leq  {\Vert \psi_0 \Vert }^2_{F_0 \otimes F_0} <  \infty.
\end{split}
\end{equation*} 
This completes the proof.        
\end{proof}

Theorem \ref{one} is stated in a general form for  GRFs  with stationary increment, with no restriction on their spectral measures. 
However, verifying the second condition in this theorem, which involves finding a function $ \psi \in \mathcal{L}_D(F_0) \otimes 
\mathcal{L}_D(F_0)$ with the property \eqref{ess}, seems to be hard. If we put the condition $ \mbox{(C1)} $ on one of  the 
spectral measures (say, $F_0$),  we get the following theorem using the reproducing kernels of 
$\mathcal{L}_{\Pi_T} (F_0) $. In fact, this theorem clarifies what must be the function $ \psi $ in Theorem \ref{one}.

\begin{thm}\label{two} 
Two centered GRFs with stationary increments, and spectral measures $ F_0 $ and $ F_1 $, with $ F_0 $ satisfying 
assumption $ (C1) $, are equivalent on $ \Pi_T $ for some $ T>0 $ if and only if:

(i) $ {\Vert \phi \Vert}_{F_0} \asymp {\Vert \phi \Vert}_{F_1}, \forall \phi \in \mathcal{L}^e_{\Pi_T} $,

(ii) $ \psi(\omega, \lambda) = K_T^0(\omega, \lambda) - \int_{\mathbb{R}^d}\!K_T^0(\omega, \gamma)
\overline{K_T^0(\lambda, \gamma)}\,F_1(d\gamma) \in \mathcal{L}_{\Pi_T} (F_0) \otimes \mathcal{L}_{\Pi_T} (F_0) $,
where $ K^0_T(. , .) $ are the reproducing kernels of the space $ \mathcal{L}_{\Pi_T}(F_0) $. 
\end{thm}

\begin{proof}
First, assume that the measures induced by them are equivalent. Then, by Theorem \ref{one}, there exists a function 
$ \psi \in \mathcal{L}_{\Pi_T} (F_0) \otimes \mathcal{L}_{\Pi_T} (F_0) $ such that  \eqref{ess} holds. 
Now, because of bilinearity and continuity of inner product together with the fact that $ \mathcal{L}_{\Pi_T}(F_0) =
 \mathcal{L}_{\Pi_T}(F_1) $, we get
\begin{equation}\label{imp}
{\langle \phi_1, \phi_2 \rangle}_{F_0} - {\langle \phi_1, \phi_2 \rangle}_{F_1} = {\langle \psi, \phi_1 \otimes \phi_2 \rangle}_{F_0 \otimes F_0}
\end{equation}
for all $ \phi_1, \phi_2 \in \mathcal{L}_{\Pi_T}(F_0) = \mathcal{L}_{\Pi_T}(F_1) $. Now, simply choose for fixed $ \omega, \lambda \in 
\mathbb{R}^{d} $, $ \phi_1(\gamma) = K_T^0(\omega,\gamma) $ and $ \phi_2(\gamma) = K_T^0(\lambda,\gamma) $, 
and replace them in \eqref{imp} to get
\begin{equation*}
\psi(\omega, \lambda) = K_T^0(\omega, \lambda) - \int_{\mathbb{R}^d}\!K_T^0(\omega, \gamma)
\overline{K_T^0(\lambda, \gamma)}\,F_1(d\gamma).
\end{equation*}
Conversely, since $ \psi \in \mathcal{L}_{\Pi_T} (F_0) \otimes \mathcal{L}_{\Pi_T} (F_0) $, by the reproducing kernel property 
we get $ \psi(\omega,\lambda) = {\langle \psi, K_T^0(\omega, . ) \otimes K_T^0(\lambda, . ) \rangle}_{F_0 \otimes F_0} $. Also, 
note that by the form of $ \psi $, we have $ \psi(\omega,\lambda) = {\langle K_T^0(\omega, . ), K_T^0(\lambda, . ) \rangle}_{F_0}
 - {\langle K_T^0(\omega, . ), K_T^0(\lambda, . ) \rangle}_{F_1} $. Combining them together, we get
\begin{equation}\label{gen}
{\langle K_T^0(\omega, . ), K_T^0(\lambda, . ) \rangle}_{F_0} - {\langle K_T^0(\omega, . ), K_T^0(\lambda, . ) \rangle}_{F_1} 
= {\langle \psi, K_T^0(\omega, . ) \otimes K_T^0(\lambda, . ) \rangle}_{F_0 \otimes F_0}.
\end{equation}
Now, since the $ span \lbrace K_T^0(\omega, . ); \omega \in \mathbb{R}^{d} \rbrace $ is dense in $ \mathcal{L}_{\Pi_T}(F_0) 
( = \mathcal{L}_{\Pi_T}(F_1) ) $, Equality \eqref{gen} holds true for all the elements in $ \mathcal{L}_{\Pi_T}(F_0) $. 
\end{proof}

Checking the first assumption in Theorem \ref{two} may not be easy in general since we need to compare the 
norms of all the elements in the space $ \mathcal{L}^e_{\Pi_T} $ under two different measures. For that purpose, 
in the following, we will find equivalent conditions which may be easier to verify in application.   

It is well known (See \cite{Dunford_1963}, p. 1009 ) that for $ \psi \in L^2(F \otimes F) $, a Hilbert-Schmidt operator on $ L^2(F) $ 
can be defined as follows
\begin{equation}\label{opr}
\left(\psi\phi\right)(\omega) = \int_{\mathbb{R}^d}\!\psi(\omega,\lambda)\phi(\lambda)\,F(d\lambda)
\end{equation}
for every $ \phi \in L^2(F) $. If we use specifically the $ \psi $ in Theorem \ref{two}, and restrict the domain to 
$\mathcal{L}_{\Pi_T} (F) $, we will have again a Hilbert-Schmidt operator on $ \mathcal{L}_{\Pi_T}(F) $. Note 
that the image of the operator $ \psi $ is in fact inside the $ \mathcal{L}_{\Pi_T}(F) $. To prove this, observe that for 
$ \phi \in \mathcal{L}_{\Pi_T}(F) $,
\[
\begin{split}
\big( \pi_{\mathcal{L}_{\Pi_T}} (\psi\phi)\big)(\omega) 
&= \int_{\mathbb{R}^d}\!(\psi\phi)(x)\overline{K_T(\omega, x)}\,F(dx)\\
&= \int_{\mathbb{R}^d}\!\int_{\mathbb{R}^d}\!\phi(y)\psi(x, y)\overline{K_T(\omega, x)}\,F(dx)F(dy)\\
&= \int_{\mathbb{R}^d}\!\int_{\mathbb{R}^d}\!\int_{\mathbb{R}^d}\!\phi(y)K_T(x, \gamma)\overline{K_T(y, \gamma)}\overline{K_T(\omega, x)}\,F(dx)F(dy)\\
&  \hspace{15em} \times (F(d\gamma) - F_1(d\gamma))\\
&= \int_{\mathbb{R}^d}\!\int_{\mathbb{R}^d}\!\phi(y)K_T(\omega, \gamma)\overline{K_T(y, \gamma)}\,(F(d\gamma) - F_1(d\gamma))F(dy)\\
&= \int_{\mathbb{R}^d}\!\phi(y)\psi(\omega, y)\,F(dy) = (\psi\phi)(\omega).
\end{split}
\]
This argument shows that $ \psi\phi \in \mathcal{L}_{\Pi_T}(F) $ for any $ \phi \in \mathcal{L}_{\Pi_T}(F) $. Also, observe that since 
$ \psi(\omega, \lambda) = \overline{\psi(\lambda, \omega)} $, the operator $ \psi $ is self-adjoint. This fact together with compactness 
of this operator (Since $ \psi $ is a Hilbert-Schmidt operator, it is already compact, see \cite{Dunford_1963}, p. 1009) enable us to use 
the Spectral Theorem for compact normal operators (See \cite{Dunford_1963}, Corollary 5, p. 905), which we will use in the proof of 
the next theorem. In fact, the next theorem is an extension of  Theorem 4.3  in \cite{VanZanten_2007}  and shows that the first condition 
in Theorem \ref{two} can be replaced by $ 1 \notin \sigma (\psi)$, 
where $ \sigma (\psi) $ is the spectrum of the operator $ \psi $. Recall that $ \sigma (\psi) $ is the set of all 
$ \lambda \in \mathbb{C} $ such that $ \lambda I - \psi $ is not an invertible operator where $ I $ is the identity operator (cf. \cite{Dunford_1963}, p. 902). 
 
 The following is an extension of Theorem 4.3  in \cite{VanZanten_2007} to the setting of random fields.   
\begin{thm}\label{HS}
Two GRFs with stationary increments and spectral measures $ F_0 $ and $ F_1 $ with $ F_0 $ satisfying the condition $ (C1) $, 
are equivalent on ${\Pi_T}$ if and only if the function defined by
\begin{equation}
\psi(\omega, \lambda) = K_T^0(\omega, \lambda) - \int_{\mathbb{R}^d}\!K_T^0(\omega, \gamma)\overline{K_T^0(\lambda, \gamma)}\,F_1(d\gamma)
\end{equation}
belongs to $ \mathcal{L}_{\Pi_T} (F_0) \otimes \mathcal{L}_{\Pi_T} (F_0) $, and $ 1 \notin \sigma (\psi)$.
\end{thm}

\begin{proof}
From \eqref{imp}, by putting $ \phi_1 = \phi_2 = \phi $, and the definition of the operator $ \psi $ in \eqref{opr}, we get
\begin{equation}\label{1}
\frac{{\Vert \phi \Vert}^2_{F_1}}{{\Vert \phi \Vert}^2_{F_0}} = 1 - \frac{{\langle \psi\phi, \phi\rangle}_{F_0}}{{\Vert \phi \Vert}^2_{F_0}}
\end{equation}
for all $ \phi \in \mathcal{L}_{\Pi_T} (F_0) $. This simply implies that first $ \sigma(\psi) \subseteq (-\infty, 1] $, and second there exists a 
finite positive constant $ C $ such that $ {\Vert \phi \Vert}_{F_1} \leq C {\Vert \phi \Vert}_{F_0}$ for all $ \phi \in \mathcal{L}^e_{\Pi_T}  $ 
since $ \psi $ is a bounded operator. This fact shows that proving $ \psi \in L^2(F_0 \otimes F_0) $ is helping us to verify half of what we 
need in the first condition of Theorem \ref{two} as well. What remains is to show that $ 1 \notin \sigma (\psi) $ if and only if there exists 
a positive constant $ c $ such that $ {\Vert \phi \Vert}_{F_0} \leq c {\Vert \phi \Vert}_{F_1}$ for all $ \phi \in \mathcal{L}^e_{\Pi_T}  $. 

First, suppose that $ {\Vert \phi \Vert}_{F_0} \leq c \, {\Vert \phi \Vert}_{F_1}$ for some $ c>0 $. If $ 1 \in \sigma (\psi) $, it means that 
there exists $ \phi \in \mathcal{L}_{\Pi_T}(F_0) $ with $ {\Vert \phi \Vert}_{F_0} = 1 $ such that  $ \psi\phi = \phi $. Putting it in \eqref{1}, 
we get $ {\Vert \phi \Vert}_{F_1} = 0 $ which is contradiction. Conversely, suppose $ 1 \notin \sigma (\psi) $, and also there exists a
sequence $ \phi_n \in \mathcal{L}^e_{\Pi_T} $ such that $ {\Vert \phi_n \Vert}_{F_0} = 1 $ for all $ n \geq 1 $, and 
$ {\Vert \phi_n \Vert}_{F_1} \rightarrow 0 $ as $ n \rightarrow \infty $. Since $ \psi $ is a self-adjoint compact operator, by Corollary 5 on
p. 905 in \cite{Dunford_1963}, there exists a countable orthonormal basis for $ \mathcal{L}_{\Pi_T}(F_0) $ consisting of eigenvectors 
of $ \psi $, denoting them by $ g_j, j=1,2, . . . $ with corresponding eigenvalues $ \lambda_j $. Now, each $ \phi_n $ has the representation 
$ \phi_n = \sum_{j=1}^{\infty} a_{nj}g_j$ for $ a_{nj} \in \mathbb{R} $. Putting this sequence  \eqref{1}, we get that $ {\langle \psi\phi_n, 
\phi_n \rangle}_{F_0} \rightarrow 1 $ which means $ \sum_{j=1}^{\infty} a^2_{nj}\lambda_j \rightarrow 1 $ as $ n \rightarrow \infty $. 
Now, since $ 1 = {{\Vert \phi_n \Vert}^2_{F_0}} = \sum_{j=1}^{\infty} a^2_{nj} $, we can rewrite the above equation as $ 0 \leq 
\sum_{j=1}^{\infty} a^2_{nj}(1 - \lambda_j) \rightarrow 0 $ (This quantity is non-negative since all the eigenvalues are bounded above 
by 1). Since $ 1 \notin \sigma (\psi) $, and $ \lbrace \lambda_j, j=1,2, . . . \rbrace $ has no accumulation points in $ \mathbb{C} $ 
except possibly $ 0 $ (See \cite{Dunford_1963}, Corollary 5, p. 905), there exists $ \epsilon>0 $ such that $ \sup \, \lbrace \lambda_j, j 
\geq 1 \rbrace  = 1 - \epsilon $. However, this implies that $ \sum_{j=1}^{\infty} a^2_{nj}(1 - \lambda_j) \geq \epsilon
 \sum_{j=1}^{\infty} a^2_{nj} = \epsilon $ for all $ n \in \mathbb{N} $, which is contradiction by the fact that this sequence must go 
 to $ 0 $ when $ n $ goes to $ \infty $. This completes the proof.   
\end{proof}

\begin{remark}\label{equivalent norms}
Notice that based on the proof of Theorem \ref{HS}, we can change the first condition in Theorem \ref{two} to\\ 

${(i)}^{\prime}$ There exists a positive constant $ c $ such that $ {\Vert \phi \Vert}_{F_0} \leq c \, {\Vert \phi \Vert}_{F_1} $ 
for all $ \phi \in \mathcal{L}^e_{\Pi_T} $. 
\end{remark}

As Lemma \ref{equivalence} emphasizes that the behavior of the spectral measure at origin does not affect the structure of the 
space $ \mathcal{L}_{\Pi_T}(F) $, one might expect the same formation in terms of the equivalence of Gaussian measures. The 
following theorem shows that changing the spectral measure on bounded subsets of $ \mathbb{R}^d $ will not affect the 
equivalence of the corresponding GRFs. In other words, for checking the equivalence of GRFs, only the behavior of their 
spectral measures at infinity is important.

\begin{thm}\label{bdd}
Suppose two GRFs with stationary increments have spectral measures $ F_0 $ and $ F_1 $ such that $ F_0 $ satisfies 
the condition $ (C1) $, and $ F_0 = F_1 \, on \, I^c $, where $ I $ is a bounded subset of $ \mathbb{R}^d $. Then, these 
two GRFs are locally equivalent.   
\end{thm}

\begin{proof}
Define $ \widetilde{F}_1(d\lambda) = {\huge 1 }_{I^c}(\lambda) F_0(d\lambda)$. First, we show that $ F_0 $ and $\widetilde{F}_1 $ 
will produce locally equivalent GRFs with stationary increments. For that, fix $ T>0 $. We will investigate the equivalence of 
measures on $ \Pi_T $. The function $ \psi $ appearing in Theorem \ref{HS} in this case is given by
\begin{equation*}
\psi(\omega, \lambda) = \int_{I}\!K_T^0(\omega, \gamma)\overline{K_T^0(\lambda, \gamma)}\,F_0(d\gamma).
\end{equation*}
Notice that by the reproducing kernel property, 
\begin{equation}\label{eq:proj}
\psi(\omega, \lambda) = \pi_{\mathcal{L}_{\Pi_T}} \left( K_T^0 (\omega, .) {\huge 1 }_{I}(.)  \right) (\lambda) =
 \overline{\pi_{\mathcal{L}_{\Pi_T}} \left( K_T^0 (\lambda, .) {\huge 1 }_{I}(.)  \right) (\omega)}
\end{equation}
The specific representation of the function $ \psi(\omega, \lambda) $ in (\ref{eq:proj}) helps us to show that 
$ \psi(\omega, \lambda) \in \mathcal{L}_{\Pi_T} (F_0) \otimes \mathcal{L}_{\Pi_T} (F_0) $. The idea of the proof is to first 
show that $ \psi(\omega, \lambda) \in L^2 \left( F_0 \otimes F_0 \right) $ and then, use the projection technique to further 
derive that $ \psi(\omega, \lambda) \in \mathcal{L}_{\Pi_T} (F_0) \otimes \mathcal{L}_{\Pi_T} (F_0) $. Note that
\begin{equation}\label{eq:new_1}
\begin{split}
{\Vert \psi \Vert}^2_{F_0 \otimes F_0}  
&= \int_{\mathbb{R}^d}\!{\left\Vert \pi_{\mathcal{L}_D(F_0)}\left( K_T^0(\omega, .) {\huge 1 }_{I}(.) \right) \right\Vert}^2_{F_0}\,F_0(d\omega)   \\
& \leq  \int_{\mathbb{R}^d}\!{\left\Vert K_T^0(\omega, .){\huge 1 }_{I}(.) \right\Vert}^2_{F_0}\,F_0(d\omega)  \\
&= \int_{\mathbb{R}^d}\!\left( \int_{\mathbb{R}^d}\!{| K^0_T(\omega,\gamma)|}^2 {\huge 1 }_{I}(\gamma) \,F_0(d\gamma) \right)  F_0(d\omega)   \\
&= \int_{\mathbb{R}^d}\!{\huge 1 }_{I}(\gamma) K^0_T(\gamma,\gamma)\,F_0(d\gamma)   \\
& \leq  C \int_{I} |\gamma|^2 \!\,F_0(d\gamma) <  + \infty.
\end{split}
\end{equation}

The second inequality in \eqref{eq:new_1} is based on Proposition \ref{zero bhv} and the fact that $ I $ is 
bounded. Now, we prove that the projection of $ \psi(\omega, \lambda) $ to the space $ \mathcal{L}_{\Pi_T} (F_0) \otimes \mathcal{L}_{\Pi_T} (F_0) $ is in fact itself. This verifies that $ \psi(\omega, \lambda) \in \mathcal{L}_{\Pi_T} (F_0) \otimes \mathcal{L}_{\Pi_T} (F_0) $. To this end, observe that

\begin{equation*}
\begin{split} 
\big( \pi_{\mathcal{L}_{\Pi_T} (F_0) \otimes \mathcal{L}_{\Pi_T} (F_0)}(\psi)\big) (\omega, \lambda) &= {\left\langle \psi, K^0_T(\omega, .) \otimes K^0_T(\lambda, .)   \right\rangle}_{F_0 \otimes F_0}\\
&= \int_{\mathbb{R}^d}\!\int_{\mathbb{R}^d}\!\psi(x, y)\overline{K^0_T(\omega, x)\overline{K^0_T(\lambda, y)}}\,F_0(dy)F_0(dx)\\
&= \int_{\mathbb{R}^d}\!\int_{\mathbb{R}^d}\!\pi_{\mathcal{L}_{\Pi_T}} \left( K_T^0 (x, .) {\huge 1 }_{I}(.)  \right) (y) \overline{K^0_T(\omega, x)\overline{K^0_T(\lambda, y)}}\\
&   \hspace{10em} \times \,F_0(dy)F_0(dx)\\
&= \int_{\mathbb{R}^d}\! \overline{\pi_{\mathcal{L}_{\Pi_T}} \left( K_T^0 (\lambda, .) {\huge 1 }_{I}(.)  \right) (x)} \, \overline{K^0_T(\omega, x)} \, F_0(dx) \\
&= \overline{\pi_{\mathcal{L}_{\Pi_T}} \left( K_T^0 (\lambda, .) {\huge 1 }_{I}(.)  \right) (\omega)} \\
&= \psi(\omega, \lambda).
\end{split}
\end{equation*}
This implies $ \psi \in \mathcal{L}_{\Pi_T}(F_0)\otimes\mathcal{L}_{\Pi_T}(F_0) $.\\


It remains to show that $ 1 \notin \sigma(\psi) $. For that purpose, take an arbitrary $ \phi \in 
\mathcal{L}_{\Pi_T}(F_0) $, and observe that (We use the fact that $ K^0_T(\omega, \lambda) 
= \overline{K^0_T(\lambda, \omega)} $, see \cite{Aronszajn_1950}, p. 344.)
\[
\begin{split}  
\left(\psi\phi\right)(\lambda) &= \int_{\mathbb{R}^d}\!\phi(\omega)\psi(\lambda,\omega)\,F_0(d\omega)\\
&= \int_{\mathbb{R}^d}\!\phi(\omega)\left( \int_{I}\!K^0_T(\lambda, \gamma)\overline{K^0_T(\omega, \gamma)}\,F_0(d\gamma)\right)\,F_0(d\omega)\\
&= \overline{\int_{I}\!\int_{\mathbb{R}^d}\!\overline{\phi(\omega)}\,\overline{K^0_T(\lambda, \gamma)}
\overline{K^0_T(\gamma, \omega)}\,F_0(d\omega)F_0(d\gamma)}\\
&= \overline{\int_{I}\!\overline{\phi(\gamma)}\,\overline{K^0_T(\lambda, \gamma)}\,F_0(d\gamma)}\\
&= \overline{\pi_{\mathcal{L}_{\Pi_T}(F_0)}\left( \overline{\phi}{\huge 1 }_{I}\right)(\lambda)}.
\end{split}
\]
Therefore, if $ \psi\phi = \phi $, it implies in particular that $ {\Vert \phi \Vert}_{F_0} \leq {\Vert \overline{\phi} {\huge 1}_{I} \Vert}_{F_0} $. 
This means $ \phi = 0 $ almost everywhere with respect to $ F_0 $ in $ I^c $. Hence, since $ \phi $ is an entire function, this implies that $ \phi = 0 $. 
Thus, $ 1 $ cannot be in the spectrum of $ \psi $.\\
So far, we proved GRFs with spectral measures $ F_0 $ and $ \widetilde{F}_1 $ are locally equivalent, but since $ F_0 = F_1 $ on $ I $, 
similarly, we can say that $ F_1 $ and $ \widetilde{F}_1 $ produce locally equivalent GRFs. Putting these two together, we get the desired result.         
\end{proof}

Theorems \ref{two} and \ref{HS} give necessary and sufficient conditions for equivalence of GRFs with stationary increments, 
but  it might be difficult to verify the conditions in these theorems. In the literature, 
there are sufficient conditions for equivalence of certain GRFs in terms of their spectral densities. These conditions 
are easily verifiable once the two spectral densities are known. For example, we refer to \cite{Ibragimov_Rozanov_1978}, 
Theorem 17, p. 104, \cite{Skorokhod_Yadrenko_1973}, Theorem 4, and \cite{Yadrenko_1983}, theorem 4, p. 156 for 
stationary GRFs; and to \cite{VanZanten_2007,VanZanten_2008}, and \cite{Stein_2004} for some nonstationary cases. 
The following is the main theorem of this paper which gives an explicit sufficient condition in terms of the spectral measures
for the equivalence of GRFs with stationary increments.

\begin{thm}\label{almost}
Suppose that the spectral measure $ F_0 $ and $ F_1 $ have positive densities $ f_0 $ and $ f_1 $ with respect to the 
Lebesgue measure, and $ F_0 $ satisfies the condition $ (C1) $. If there exists a finite constant $ C > 0 $ such that 
$ {\Vert \phi \Vert}_{F_0} \leq C {\Vert \phi \Vert}_{F_1} $ for all $ \phi \in \mathcal{L}^e_{\Pi_T} $, and 
\begin{equation}\label{Int}
\int_{\vert \lambda \vert > k}\!{\left( \frac{f_1(\lambda) - f_0(\lambda)}{f_0(\lambda)}\right)}^2 K^0_T(\lambda,\lambda)
f_0(\lambda)\,d\lambda < \infty
\end{equation}
for some $ k > 0 $, then GRFs with stationary increments and spectral measures $ F_0 $ and $ F_1 $ are equivalent 
on $ {\Pi_T} $.  
\end{thm}
\begin{proof}
Applying Theorem \ref{bdd}, we can change the value of $ f_1 $ on any bounded set, without having any consequences 
on the equivalence. So, we assume here that $ f_0 = f_1 $ on $ \vert \lambda \vert \leq k $. The function $ \psi $ in Theorem 
\ref{HS} here will be of the form
\[
\begin{split}
\psi(\omega,\lambda) &= \int_{\mathbb{R}^d}\!K_T^0(\omega, \gamma)\overline{K_T^0(\lambda, \gamma)}
\bigg( f_0(\gamma) - f_1(\gamma) \bigg)\,d\gamma\\
&= \pi_{\mathcal{L}_{\Pi_T}(F_0)}\left( K_T^0(\omega, .)\frac{f_0 - f_1}{f_0}\right)(\lambda). 
\end{split}
\]
(Since $ {\left| K^0_T(\omega,\lambda)\right|}^{2} \leq K^0_T(\omega,\omega) K^0_T(\lambda, \lambda) $, 
\eqref{Int} implies that $ K_T^0(\omega, .)\frac{f_0 - f_1}{f_0}$ $ \in L^2(F_0)$ for all $ \omega \in \mathbb{R}^d $. Hence 
using the orthogonal projection is eligimate). Now, it follows that 
\[
\begin{split}
\int_{\mathbb{R}^d} \int_{\mathbb{R}^d}\!{\left| \psi(\omega,\lambda)\right|}^{2}\,F_0(d\lambda)F_0(d\omega)
&= \int_{\mathbb{R}^d}\!{\left\Vert \pi_{\mathcal{L}_D(F_0)}\left( K_T^0(\omega, .)\frac{f_0 - f_1}{f_0}\right) \right\Vert}^2_{F_0}\,F_0(d\omega)\\
& \leq  \int_{\mathbb{R}^d}\!{\left\Vert K_T^0(\omega, .)\frac{f_0 - f_1}{f_0} \right\Vert}^2_{F_0}\,F_0(d\omega)\\
&= \int_{\mathbb{R}^d}\!\left( \int_{\mathbb{R}^d}\!{| K^0_T(\omega,\gamma)|}^2 {\left( \frac{f_0(\gamma) - f_1(\gamma)}{f_0(\gamma)}\right)}^2 
\,F_0(d\gamma) \right) F_0(d\omega)\\
&= \int_{\mathbb{R}^d}\!{\left(\frac{f_0(\gamma) - f_1(\gamma)}{f_0(\gamma)}\right)}^2 K^0_T(\gamma,\gamma)f_0(\gamma)\,d\gamma\\
&= \int_{\vert \gamma \vert > k}\!{\left(\frac{f_1(\gamma) - f_0(\gamma)}{f_0(\gamma)}\right)}^2 K^0_T(\gamma,\gamma)f_0(\gamma)\,d\gamma < \infty
\end{split}
\]
by the integrability assumption \eqref{Int}. Hence $ \psi \in L^2(F_0 \otimes F_0) $. 

Now, we apply similar arguments as in the proof of Theorem \ref{bdd} to show that in fact $ \psi \in 
\mathcal{L}_{\Pi_T} (F_0) \otimes \mathcal{L}_{\Pi_T} (F_0) $. To this end, observe that
\[
\begin{split}
\big( \pi_{\mathcal{L}_{\Pi_T} (F_0) \otimes \mathcal{L}_{\Pi_T} (F_0)}(\psi)\big) (\omega, \lambda) 
&= {\left\langle \psi, K^0_T(\omega, .) \otimes K^0_T(\lambda, .)   \right\rangle}_{F_0 \otimes F_0}\\
&= \int_{\mathbb{R}^d}\!\int_{\mathbb{R}^d}\!\psi(x, y)\overline{K^0_T(\omega, x)\overline{K^0_T(\lambda, y)}}\,F_0(dy)F_0(dx)\\
&= \int_{\mathbb{R}^d}\!\int_{\mathbb{R}^d}\!\pi_{\mathcal{L}_{\Pi_T}(F_0)}\bigg( K_T^0(x, \cdot)\frac{f_0 - f_1}{f_0}\bigg)(y)
 \overline{K^0_T(\omega, x)\overline{K^0_T(\lambda, y)}}\\
&   \hspace{10em} \times \,F_0(dy)F_0(dx)\\
&= \int_{\mathbb{R}^d}\! \overline{\pi_{\mathcal{L}_{\Pi_T}(F_0)}\bigg( K_T^0(\lambda, \cdot)\frac{f_0 - f_1}{f_0}\bigg)(x)} 
\, \overline{K^0_T(\omega, x)} \, F_0(dx) \\
&= \overline{\pi_{\mathcal{L}_{\Pi_T}(F_0)}\bigg( K_T^0(\lambda, \cdot)\frac{f_0 - f_1}{f_0}\bigg)(\omega)} \\
&= \pi_{\mathcal{L}_{\Pi_T}(F_0)}\bigg( K_T^0(\omega, \cdot)\frac{f_0 - f_1}{f_0}\bigg)(\lambda) = \psi(\omega, \lambda).
\end{split}
\]
This completes the proof. 
\end{proof}

In  \eqref{Int}, in addition to the behavior of the spectral densities at infinity, the growth rate of the diagonal 
elements of the reproducing kernels of the space $ \mathcal{L}_{\Pi_T}(F_0) $ at infinity also plays an 
important role. Since finding explicit forms of reproducing kernels are difficult, we need at least to find 
upper bounds for the growth rate of the diagonal terms. The following condition on spectral density 
helps us to accomplish this task:

\begin{verse}
(C2) \, For spectral density $ f $, there exist an entire function $ \phi_0 $ on $ \mathbb{C}^d $ of finite 
exponential type such that $ f(\lambda) \asymp {\left| \phi_0(\lambda) \right|}^2 $ as $ \vert \lambda \vert 
\rightarrow \infty $ on $ \mathbb{R}^d $.
\end{verse}

The following lemma  shows that  under (C2)  we have an upper bound for the behavior of the reproducing 
kernels on the diagonal at infinity.

\begin{lemma}\label{exp}
Suppose $ f_0 $ is a spectral density such that it satisfies $ (C1) $ and $ (C2) $ for some entire function 
$ \phi_0 $. Then, for $ T>0 $, there exists a finite constant $C > 0$ such that the reproducing kernel 
$ K^0_T $ of $ \mathcal{L}_{\Pi_T}(f_0) $ satisfies
\begin{equation*}
{\left| K^0_T(\omega,\lambda)\right|}^2 \leq C \frac{K^0_T(\omega,\omega)}{f_0(\lambda)} 
\end{equation*}
for all  $ \omega,\,   \lambda \in \mathbb R^d $ with $|\lambda|$ large enough. In particular,
\begin{equation}
\left| K^0_T(\lambda,\lambda)\right| \leq  \frac{C}{f_0(\lambda)} 
\end{equation}
for all $ \lambda \in \mathbb R^d $ with $|\lambda|$ large enough.   
\end{lemma}
\begin{proof}
The idea of the proof is similar to the one in Lemma 3 in \cite{VanZanten_2008}. Put 
$ f(\lambda) = {\left| \phi_0(\lambda) \right|}^2 $. Since $ f $ and $ f_0 $ are comparable 
at $ \infty $, and $ f $ is bounded around $ 0 $, it is clear that $ f $ is satisfying both 
conditions \eqref{Con} and (C1). This means we can define $ \mathcal{L}_{\Pi_T}(f) $ 
the same way, and this space is also a RKHS. Consider an arbitrary orthonormal basis 
for this space, and denote them by $ \psi_k , k = 1, 2, . . . \,$. Now, by Lemma \ref{Pitt}, 
they are entire functions on $ \mathbb{C}^d $ with finite exponential type which doesn't 
depend on $ k $. Further, we know $ \psi_k \phi_0 \in L^2(\mathbb{R}^d) $ since 
\begin{equation*}
\int_{\mathbb{R}^d}\!{\vert \psi_k(\lambda) \phi_0(\lambda)\vert}^2\,d\lambda = 
\int_{\mathbb{R}^d}\!{\vert \psi_k(\lambda) \vert}^2 f(\lambda)\,d\lambda = 1 < \infty.
\end{equation*}
Therefore, we can apply the Paley-Wiener Theorem (\cite{Ronkin_1974}, Theorem 3.4.2, p. 171) 
to get $ \psi_k\phi_0 = \widehat{g}_k $ for certain functions $ g_k \in L^2(B) $ where $ B $ is 
a bounded subset of $ \mathbb{R}^d $ (Here $ \widehat{h} $ stands for the Fourier transform 
of $ h $). By Parseval's identity, we can deduce that $ g_k $'s are orthonormal in $ L^2(B) $. It 
follows from Bessel's inequality that 
\[
\begin{split}
\sum_{k}{\vert\psi_k(\lambda)\vert}^2 f(\lambda)
&= \sum_{k} {\left\vert \int_{B}\!e^{-i\langle t, \lambda \rangle}g_k(t)\,dt  \right\vert}^2\\
& \leq  \int_{B}\!{\left\vert e^{i\langle t, \lambda \rangle}\right\vert}^2\,dt = m(B), 
\end{split} 
\]
where $ m(B) $ is the Lebesgue measure of $ B $. Therefore,
\begin{equation*}
 \sum_{k}{\vert\psi_k(\lambda)\vert}^2 \leq m(B)/f(\lambda) \leq C/f_0(\lambda)
\end{equation*}
for  $ \lambda \in \mathbb R^d$ with $|\lambda|$ large enough. Now, consider the reproducing kernels of 
$ \mathcal{L}_{\Pi_T}(f_0) $, denoting them by $ K^0_T( \omega, .) $. Since $ f(\lambda) \asymp f_0(\lambda) $ 
as $ |\lambda| \rightarrow \infty $, by Lemma \ref{equivalence}, $ K^0_T( \omega, .) $ belong to 
$ \mathcal{L}_{\Pi_T}(f) $ as well for ll $ \omega \in \mathbb{R}^d $. Thus, we can expand it using the 
basis $ \psi_k $, and get $ K^0_T(\omega, \lambda) = \sum_{k}{\langle K^0_T(\omega, .), \psi_k \rangle}_{f} 
\psi_k(\lambda) $, and then by Cauchy-Schwarz and Lemma \ref{equivalence}, we get 
\begin{eqnarray*}
{\vert K^0_T(\omega, \lambda) \vert}^2
& \leq & {\Vert K^0_T(\omega, .) \Vert}^2_f\,\sum_{k}{\vert\psi_k(\lambda)\vert}^2\\
& \leq & c \, {\Vert K^0_T(\omega, .) \Vert}^2_{f_0}\,\sum_{k}{\vert\psi_k(\lambda)\vert}^2\\
&=&  c \, K^0_T(\omega, \omega) \,\sum_{k}{\vert\psi_k(\lambda)\vert}^2, 
\end{eqnarray*}
which makes the proof complete.
\end{proof}

Theorem \ref{almost} in combination with Lemma \ref{exp} leads to an appealing result. 
If the relative difference between two spectral densities is square integrable at infinity, then 
the corresponding GRFs with stationary increments will be locally equivalent. We finish this 
section by proving this fact.

\begin{thm}\label{main}
Suppose that the spectral measures $ F_0 $ and $ F_1 $ have positive densities $ f_0 $ and 
$ f_1 $ with respect to the Lebesgue measure, with  $ f_0 $ satisfying (C1)  and (C2) 
for some entire function $ \phi_0 $ on $ \mathbb{C}^d $ of finite exponential type. If there exists 
a finite constant $ k > 0 $ such that 
\begin{equation}\label{Int1}
\int_{\vert \lambda \vert > k}\!{\left( \frac{f_1(\lambda) - f_0(\lambda)}{f_0(\lambda)}\right)}^2\,d\lambda < \infty,
\end{equation}
then GRFs with stationary increments having spectral measures $ F_0 $ and $ F_1 $ are locally equivalent. 
\end{thm}
\begin{proof}
Thanks to Lemma \ref{exp} and Theorem \ref{almost}, it is sufficient to prove that $ 1 \notin \sigma(\psi) $. 
In spirit of Theorem \ref{bdd}, we can assume that $ f_0 = f_1 $ on 
$ \vert \lambda \vert \leq k $. Now, take an arbitrary element $ \phi \in \mathcal{L}_D(f_0) $, and 
observe that by using the multidimensional Paley-Wiener Theorem (\cite{Ronkin_1974}, Theorem 
3.4.2. p. 171), we derive that $ \phi \phi_0 $ is the inverse Fourier transform of a squared integrable 
function  $g$ with bounded support, $ B $ in $ \mathbb{R}^d $. This implies that 
\[
\begin{split}
{\left| \phi(\lambda)\phi_0(\lambda)\right|}^2 &= {\left| \int_{B}\!e^{-i\langle\lambda,\gamma\rangle}g(\gamma)\,d\gamma \right|}^2\\
& \le  \int_{B}\!g^2(\lambda)\,d\lambda <  \infty
\end{split}
\]
for all $ \lambda \in \mathbb{R}^d $. This means $ \phi\phi_0 $ is bounded on $ \mathbb{R}^d $. This fact 
together with \eqref{Int1} imply that $ \phi \frac{f_0 - f_1}{f_0} \in L^2(f_0) $. Now, observe that
\[
\begin{split}
\left(\psi\phi\right)(\lambda) &= \int_{\mathbb{R}^d}\!\phi(\omega)\psi(\lambda, \omega)\,F_0(d\omega)\\
&= \int_{\mathbb{R}^d}\!\phi(\omega)\left( \int_{\vert\gamma\vert > k}\!K_T^0(\lambda, \gamma)
\overline{K_T^0(\omega, \gamma)}\,\left(\frac{f_0(\gamma) - f_1(\gamma)}{f_0(\gamma)}\right)\,F_0(d\gamma) \right)\,F_0(d\omega)\\
&= \overline{\int_{\vert\gamma\vert > k}\!\int_{\mathbb{R}^d}\!\overline{\phi(\omega)}\,\overline{K_T^0(\lambda, \gamma)}
\overline{K_T^0(\gamma, \omega)}\left(\frac{f_0(\gamma) - f_1(\gamma)}{f_0(\gamma)}\right)\,F_0(d\omega)F_0(d\gamma)}\\
&= \overline{\int_{\vert\gamma\vert > k}\!\overline{\phi(\gamma)} \left(\frac{f_0(\gamma) - f_1(\gamma)}{f_0(\gamma)}\right)
 \overline{K_T^0(\lambda, \gamma)}\,F_0(d\gamma)}\\
&= \overline{\pi_{\mathcal{L}_{\Pi_T}(F_0)}\left( \overline{\phi} \frac{f_0 - f_1}{f_0}{\huge 1 }_{\lbrace\vert\gamma\vert > k\rbrace}\right)(\lambda)}.
\end{split}
\]
Now, similar to the proof of Theorem \ref{bdd}, if $ \psi \phi = \phi $, we get that $ {\| \phi \|  }_{F_0} \leq
 {\left\Vert \overline{\phi} \frac{f_0 - f_1}{f_0}{\huge 1 }_{\lbrace\vert\gamma\vert > k\rbrace} \right\Vert  }_{F_0} $. 
Letting $ k \rightarrow \infty $, by the Dominated Convergence Theorem, we get $ {\| \phi \|  }_{F_0} = 0 $. This 
implies that $ \phi = 0 $. Thus, $ 1 $ cannot be in the spectrum of $ \psi $ and this concludes proof.       
\end{proof}

\section{Application}

In this section, we apply the results in Section \ref{main results} to some anisotropic GRFs with stationary 
increments. In particular, we consider GRFs with stationary increments and spectral density of the form
\begin{equation}\label{anisotropic spectrals}
f(\lambda) = \frac{1}{\left( \sum_{j=1}^{d} |{\lambda}_j|^{\beta_j} \right)^\gamma},
\end{equation}
where $ \lambda = (\lambda_1, . . . , \lambda_d) \in \mathbb{R}^d \backslash \lbrace 0 \rbrace $, $ \beta_j>0 $ for 
all $ j = 1, . . ., d $, and $ \gamma > \sum_{j=1}^{d} \frac{1}{\beta_j}$. The latter condition guaranties the integrability 
condition in \eqref{Con} for spectral measures (See Proposition 2.1 in \cite{Xue_Xiao_2011}). Fractal and smoothness 
properties of this family of GRFs are discussed in \cite{Xue_Xiao_2011}. Now we apply Theorem 
\ref{main} to determine the equivalence of Gaussian measures induced by these GRFs. 

To this end, first notice that  (C1) is obviously satisfied for spectral densities of the form \eqref{anisotropic spectrals}. 
The next lemma shows that these spectral densities also satisfy (C2).

\begin{lemma} \label{Lem:41}
Spectral density functions of the form \eqref{anisotropic spectrals} satisfy condition $ (C2) $. 
\end{lemma}
\begin{proof}
First of all, it is obvious that 
\begin{equation*}
\frac{1}{\left( \sum_{j=1}^{d} |{\lambda}_j|^{\beta_j} \right)^\gamma}  \asymp \frac{1}{\left( 1 + \sum_{j=1}^{d} |{\lambda}_j|^{\beta_j} \right)^\gamma},
\end{equation*}
as $ \vert \lambda \vert \rightarrow \infty $. Therefore, it suffices to prove the lemma for functions of the form on the right hand side. 
Now, similar to the construction made in the proof of Lemma 2.3 in \cite{Luan_Xiao_2012}, we can find a function $ \phi \in L^2(B) $ for 
some bounded subset $ B \subseteq \mathbb{R}^d $ such that 
\begin{equation*}
\frac{1}{\left( 1 + \sum_{j=1}^{d} |{\lambda}_j|^{\beta_j} \right)^\gamma}  \asymp {\big\vert \widehat{\phi}(\lambda) \big\vert}^2,
\end{equation*}
as $ \vert \lambda \vert \rightarrow \infty $, which $ \widehat{\phi} $ is the Fourier transform of $ \phi $. By the Paley-Wiener Theorem 
(\cite{Ronkin_1974}, Theorem 3.4.2. p. 171), $ \widehat{\phi} $ is actually the restriction on $ \mathbb{R}^d $ of an entire function 
on $ \mathbb{C}^d $ with finite exponential type. This finishes the proof.
\end{proof}

The next theorem proves that, under certain conditions, the mixture of spectral densities 
of the form \eqref{anisotropic spectrals} will be equivalent to the one with the lowest decay rate at infinity. Similar 
results of this type have been proved  by \cite{VanZanten_2007} and \cite{Cheridito_2001} for linear combinations 
of independent fractional Brownian motions.

\begin{thm}\label{mixed_1}
Suppose $ X $ and $ Y $ are two independent centered GRFs with stationary increments with spectral densities 
of the form \eqref{anisotropic spectrals} with parameters $ (\beta_j, \gamma) $ and $ (\beta^{\prime}_j, \gamma^{\prime}) $, 
respectively. Then, if
\begin{equation}\label{mixed}
\gamma^{\prime} > \frac{1}{2} \sum_{j=1}^{d} \bigg( \frac{1}{\beta^{\prime}_j} \bigg) + \max_{1 \leq j \leq d} 
\bigg\lbrace \frac{\beta_j}{\beta^{\prime}_j} \bigg \rbrace \gamma,
\end{equation}
then $ X $ and $ X+Y $ are locally equivalent. 
\end{thm}

\begin{proof}
Using Lemma \ref{Lem:41} and Theorem \ref{main}, we only need to show 
\begin{equation}\label{multi}
\int_{|\lambda|>1}\!\frac{ {\left( \sum_{j=1}^{d} {|\lambda_j|}^{\beta_j} \right)}^{2\gamma }   }{ {\left( \sum_{j=1}^{d} 
{|\lambda_j|}^{\beta^{\prime}_j} \right)}^{2\gamma^{\prime}} }\,d\lambda < \infty.
\end{equation}

By using the inequality $ {(a+b)}^p \leq 2^p(a^p+b^p) $, we can break the integral in \eqref{multi} into $ d $ integrals. Thus,
 it's enough to show for each fixed $ j = 1, ..., d $
\begin{equation}\label{single}
I_j := \int_{|\lambda|>1}\!\frac{ |\lambda_j|^{2\beta_j \gamma} } { {\left( \sum_{k=1}^{d} {|\lambda_k|}^{\beta^{\prime}_k} \right)}^{2\gamma^{\prime}} }\,d\lambda < \infty.
\end{equation} 
Since $ |\lambda|>1 $, this implies that $ |\lambda_k|>1/\sqrt{d} $ for some $ k \in \lbrace 1, . . ., d\rbrace $. We distinguish two cases: 
Case I is when $ k=j $, and Case II is when $ k \neq j $. In both cases, we use the following fact that,  for positive constants 
$ \beta $ and $ \gamma $, and nonnegative constant $ b$, there exists a finite positive constant $ c $ such that for all $ a>0 $
\begin{equation}\label{integral}
\begin{split}
\int_{0}^{\infty}\!\frac{x^b}{{\left( a + x^\beta \right)}^\gamma}\,dx
&= a^{-\left( \gamma - \frac{1}{\beta} - \frac{b}{\beta} \right)}\int_{0}^{\infty}\!\frac{y^b}{{\left( 1 + y^\beta \right)}^\gamma}\,dy\\
&=  \begin{cases} c\, a^{-\left( \gamma - \frac{1}{\beta} - \frac{b}{\beta} \right)} &\mbox{if } \beta \gamma - b > 1, \\ \infty 
& \mbox{if } \beta \gamma - b \leq 1. \end{cases} 
\end{split}
\end{equation}

First, let's consider case I. By applying  \eqref{integral} $d$ times, we have
\[
\begin{split}
I_j & \leq  \int_{\frac{1}{\sqrt{d}}}^{\infty}\! {|\lambda_j|}^{2\beta_j\gamma}\,\underbrace{\int_{0}^{\infty}\,. . .\,
\int_{0}^{\infty}}_{d - 1}\!\frac{d\lambda}{{\left( \sum_{r=1}^{d} {|\lambda_r|}^{\beta^{\prime}_r} \right)}^{2\gamma^{\prime}}}\\
& \leq  c^{\prime}\,\int_{\frac{1}{\sqrt{d}}}^{\infty}\!\frac{{|\lambda_j|}^{2\beta_j\gamma}}{\left( {|\lambda_j|}^{\beta^{\prime}_j}\right)^{2\gamma^{\prime} - \sum_{r \neq j}\frac{1}{\beta^{\prime}_r} }}\,d\lambda_j < \infty 
\end{split}
\]  
since $ \beta^{\prime}_j\left( 2\gamma^{\prime} -  \sum_{r \neq j}\frac{1}{\beta^{\prime}_r}  \right) > 2\beta_j \gamma + 1 $ due to 
condition \eqref{mixed}.

Next, we consider case II, where $ k \neq j $. Similar to case I, we use \eqref{integral} iteratively, but we 
take the integration in different order. Denote the integration in $ \lambda_i $ for $ i \neq j, k $ by $ d{\lambda}_{\backslash j,k } $, and observe
\[
\begin{split}
I_j &\leq  \int_{\frac{1}{\sqrt{d}}}^{\infty}\! d\lambda_k\,\underbrace{\int_{0}^{\infty}\,. . .\,\int_{0}^{\infty}}_{d - 1}\!\frac{{|\lambda_j|}^{2\beta_j\gamma}}{{\left( \sum_{r=1}^{d} {|\lambda_r|}^{\beta^{\prime}_r} \right)}^{2\gamma^{\prime}}}\,d\lambda_j d{\lambda}_{\backslash j,k }\\
&\leq  c\,\int_{\frac{1}{\sqrt{d}}}^{\infty}\! d\lambda_k\,\underbrace{\int_{0}^{\infty}\,. . .\,\int_{0}^{\infty}}_{d - 2}\!\frac{1}{{\left( \sum_{r=1}^{d} {|\lambda_r|}^{\beta^{\prime}_r} \right)}^{2\gamma^{\prime} - \frac{1}{\beta^{\prime}_j} - \frac{2\beta_j \gamma}{\beta^{\prime}_j} }}\, d{\lambda}_{\backslash j,k }\\
&\leq  c^{\prime}\,\int_{\frac{1}{\sqrt{d}}}^{\infty} \frac{1}{\left( {|\lambda_k|}^{\beta^{\prime}_k}\right)^{2\gamma^{\prime} -  \frac{2\beta_j \gamma}{\beta^{\prime}_j} - \sum_{r \neq k}{\frac{1}{\beta^{\prime}_r} }}}\,d\lambda_k < \infty, 
\end{split}
\]  
where the second and the third inequality follow since $ 2\gamma^{\prime}\beta^{\prime}_j > 2\beta_j\gamma + 1 $ and $ \beta^{\prime}_k \left( 2\gamma^{\prime} - \frac{2\beta_j \gamma}{\beta^{\prime}_j} - \sum_{r \neq k}\frac{1}{\beta^{\prime}_r}  \right) > 1 $, respectively using the assumption \eqref{mixed}. This finishes the proof. 
\end{proof}

Next, we consider a similar situation as in Theorem \ref{mixed_1}, but this time we put discrete spectral measure mixed with 
the ones of the form \eqref{anisotropic spectrals}. For that purpose, consider discrete spectral measure of the form
\begin{equation}\label{discrete}
F\left( \lbrace -\gamma^n \rbrace \right) = F\left( \lbrace \gamma^n \rbrace \right) = \alpha_n,
\end{equation}
where $ \gamma^n \in \mathbb{R}^d $, $ \alpha_n \geq 0 $, for $ n \geq 1 $, and $ \sum_{n=1}^{\infty} \frac{{|\gamma^n|}^2}
{1 + {|\gamma^n|}^2} \alpha_n < \infty $. If $ \lbrace \gamma^n, n=1,2,... \rbrace $ is a bounded subset of $ \mathbb{R}^d $, 
then in view of Theorem \ref{bdd}, this spectral measure will not affect the equivalence of Gaussian measures. Therefore, 
we consider here only the case where $ |\gamma^n| \rightarrow \infty $ as $ n \rightarrow \infty $. 

\begin{thm}\label{mixed_2}
Let $ X $ and $ Y $ be two independent centered GRFs with stationary increments with spectral measures $ F_X $ and $ F_Y $. 
Suppose $ F_X $ has density with respect to Lebesgue measure on $ \mathbb{R}^d $, denoted by $ f $, which satisfies both 
conditions $ (C1) $ and $ (C2) $, and $ F_Y $ is a discrete measure of the form \eqref{discrete}. Then, if  
\begin{equation}\label{sum}
\sum_{n>N} \frac{\alpha_n}{f(\gamma^n)} < \infty,
\end{equation}
for some $ N \geq 1 $, then $ X $ and $ X + Y $ are locally equivalent.
\end{thm}
\begin{proof}
First of all, using Theorem \ref{bdd}, we can assume $ \alpha_n = 0 $ for $ n = 1, . . ., N $, 
without having any consequences on the equivalence of Gaussian measures. Second, observe 
that for all $ \phi \in \mathcal{L}^e_{\Pi_T}$, $ {\Vert \phi \Vert}_{F_X} \leq {\Vert \phi \Vert}_{F_X + F_Y} $, 
which by Remark \ref{equivalent norms}, is equivalent to condition (i) in Theorem \ref{two}. 

All we need to prove is then to show that the function $ \psi $ in Theorem \ref{two} is in $ \mathcal{L}_{\Pi_T} (f) 
\otimes \mathcal{L}_{\Pi_T} (f) $. For that, note that the function $ \psi $ can be written as 

\begin{equation}\label{eq:psi}
\psi(\omega, \gamma) = \sum_{n>N} \alpha_n K_T(\omega, \gamma^n)\overline{K_T(\lambda, \gamma^n)}.
\end{equation}

Observe that functions of the form $ K_T(., \gamma^n)\overline{K_T(., \gamma^n)} $ belong to the space 
$ \mathcal{L}_{\Pi_T}^e  \otimes \mathcal{L}_{\Pi_T}^e $ since $ K_T(., \gamma^n) $'s are the reproducing 
kernel elements of $ \mathcal{L}_{\Pi_T}^e $. Therefore, if we can show that $ {\Vert \psi \Vert}^2_{f \otimes f} 
< + \infty $, it implies that the function $ \psi $ defined in \eqref{eq:psi} is the $ L^2(f \otimes f) $ limit of the 
partial sums and hence, $ \psi \in \mathcal{L}_{\Pi_T} (f) \otimes \mathcal{L}_{\Pi_T} (f) $ since $\mathcal{L}_{\Pi_T} (f) 
\otimes \mathcal{L}_{\Pi_T} (f) $ is the closure in $ L^2(f \otimes f) $ of the space $ \mathcal{L}_{\Pi_T}^e  \otimes 
\mathcal{L}_{\Pi_T}^e $. To this end, observe that
\[
\begin{split}
{\Vert \psi \Vert}_{f \otimes f}
& \leq  \sum_{n > N}  {\Vert \alpha_n K_T(., \gamma^n)\overline{K_T(., \gamma^n)} \Vert}_{f \otimes f} \\
&= \sum_{n > N} \alpha_n K_T \left( \gamma^n, \gamma^n \right)  \\
&\leq  C  \sum_{n>N} \frac{\alpha_n}{f(\gamma^n)} < + \infty,
\end{split}
\]
by the assumption \eqref{sum}. The proof is complete.
\end{proof}  


Finally, we give another application of Theorem \ref{main}.  We consider the spectral densities of the following form 
\begin{equation}\label{anisotropic spectrals_2}
f(\lambda) = \frac{1}{\left( \sum_{j=1}^{d} |{\lambda}_j|^{H_j} \right)^{Q+2}},
\end{equation}
where $ \lambda = (\lambda_1, . . . , \lambda_d) \in \mathbb{R}^d \backslash \lbrace 0 \rbrace $, 
$ 0<H_j<1 $ for all $ j = 1, . . ., d $, and $ Q = \sum_{j=1}^{d} \frac{1}{H_j}$.  According to Remark 2.2 in \cite{Xue_Xiao_2011}, 
every positive function of the form \eqref{anisotropic spectrals} is comparable to a function of the form \eqref{anisotropic spectrals_2} 
as $|\lambda| \to \infty$. See \cite{Xue_Xiao_2011} for the explicit relationship between the parameters $(\beta_1, \ldots, \beta_d, \gamma)$ 
in \eqref{anisotropic spectrals}
and $(H_1, \ldots, H_d)$ in \eqref{anisotropic spectrals_2}.  \cite{Xue_Xiao_2011} proved that the smoothness and fractal properties 
of a Gaussian random field with spectral density \eqref{anisotropic spectrals} is characterized by the corresponding parameters  $(H_1, \ldots, H_d)$.
The following theorem shows that a similar phenomenon occurs for equivalence of  these Gaussian random fields.

\begin{thm}
Suppose $ f_0 $ and $ f_1 $ are spectral densities of the form \eqref{anisotropic spectrals_2} with parameters 
$ H_j^{0} $ and $ H_j^{1} $ ($ j = 1, . . ., d $), respectively. Then, GRFs with stationary increments and spectral densities 
$ f_0 $ and $ f_1 $ are locally equivalent if and only if $ H_j^0 = H_j^1 $ for all $ j = 1, . . ., d$. 
\end{thm}  
\begin{proof} The sufficiency is obvious, so we only need to prove the necessity. 
Suppose for some $ k \in \lbrace 1, ..., d\rbrace $, $ H_k^0 < H_k^1 $. By Lemma 3.2 in \cite{Xue_Xiao_2011}, 
there exist $ c_1, c_2 > 0 $, 
such that for all $ t \in \mathbb{R}^d $
\begin{equation}\label{eq:norms_compare}
c_1 \sum_{j=1}^{d} {\vert t \vert}^{2H_j^i} \leq {\Vert e_t \Vert}^2_{f_i} \leq c_2 \sum_{j=1}^{d} {\vert t \vert}^{2H_j^i},
\end{equation}
for $ i = 0, 1 $. If we simply choose $ t \in \mathbb{R}^d $ with $ t_k = l $, and $ t_j = 0 $ for $ j \neq k $, we get
\begin{equation*}
\frac{{\Vert e_t \Vert}^2_{f_1}}{{\Vert e_t \Vert}^2_{f_0}} \leq \frac{c_2}{c_1} l^{2 \left( H_k^1 - H_k^0 \right)} 
\rightarrow 0 \hspace{1em} as \hspace{0.5em} l \rightarrow 0.
\end{equation*}
This violates the necessary condition for equivalence of Gaussian measures in Theorem \ref{one}.
\end{proof}


\bibliography{Safikhani_Reference}
\end{document}